\RequirePackage[l2tabu, orthodox]{nag}
\RequirePackage{fixltx2e}

\documentclass[11pt]{amsart}
\usepackage{stmaryrd}
\usepackage{amsmath}
\usepackage{amsfonts}
\usepackage{amsthm}
\usepackage{mathtools}
\usepackage{txfonts}
\usepackage[usenames,dvipsnames]{xcolor}
\usepackage{hyperref}
\hypersetup{colorlinks=true,citecolor=NavyBlue,linkcolor=Brown,urlcolor=Orange}

\usepackage{graphicx,amssymb,amsfonts,amsmath,amscd}
\usepackage{    textcomp}
\usepackage[all,ps,cmtip]{xy}
\usepackage{amscd}
\usepackage{xspace}
\usepackage{mathrsfs}
\usepackage{ dsfont }
\usepackage[varg]{pxfonts}

\usepackage{enumerate}
\usepackage{xcolor}
\usepackage{tikz}\usetikzlibrary{decorations.markings,matrix,arrows}
\usepackage{extarrows}

\newtheorem{thm}{Theorem}[section]
\newtheorem{prop}[thm]{Proposition}
\newtheorem{lemma}[thm]{Lemma}

\newtheorem{conj}[thm]{Conjecture}

\newtheorem{def-prop}[thm]{Definition-Proposition}
\newtheorem{prop-def}[thm]{Proposition-Definition}

\theoremstyle{definition} % upshaped
\newtheorem{defi}[thm]{Definition}

\newcommand{\C}{{\bf C}}
\newcommand{\R}{{\bf R}}
\newcommand{\Z}{{\bf Z}}

\newcommand{\Q}{{\bf Q}}
\newcommand{\rX}{{\mathcal X}}

\newcommand{\ie}{\textit {i.e.}~}
\newcommand{\cf}{\textit {cf.}~}

\newcommand{\loccit}{\textit {loc.cit.}~}

\newcommand{\apriori}{\textit{a priori} }

\newcommand{\etc}{\textit {etc.}~}
\newcommand{\resp}{\textit {resp.}~}

\newcommand{\CH}{\mathop{\rm CH}\nolimits} % Chow groups
\newcommand{\dual}{\mathop{^\vee}\nolimits} % dual
\newcommand{\End}{\mathop{\rm End}\nolimits} %endomorphism ring
\newcommand{\Hom}{\mathop{\rm Hom}\nolimits}
\newcommand{\id}{\mathop{\rm id}\nolimits} %identity
\newcommand{\im}{\mathop{\rm Im}\nolimits} % image, imaginary part is \Im
\renewcommand{\P}{\mathop{\bf P}\nolimits} % projective space, projectivization
\newcommand{\pr}{\mathop{\rm pr}\nolimits} % projection
\newcommand{\rank}{\mathop{\rm rank}\nolimits}
\newcommand{\Sing}{\mathop{\rm Sing}\nolimits} % singular locus
\newcommand{\Spec}{\mathop{\rm Spec}\nolimits}
\newcommand{\SL}{\mathop{\rm SL}\nolimits}

\renewcommand{\bar}{\overline}
\newcommand{\inj}{\hookrightarrow}

\newcommand{\lra}{\xrightarrow}

\newcommand{\cart}{\ar@{}[dr]|\square} % cartesian diagrams, write it after the left-up term to produce a square in the middle of the diagram
\renewcommand{\dual}{^{\vee}} % dual
\newcommand{\isom}{\simeq} %isomorphism
\renewcommand{\tilde}{\widetilde}

\newcommand{\deff}{\mathrel{:=}}

\newcommand{\age }{\mathop{\rm {age}}\nolimits}% age
\newcommand{\h}{\mathop{\mathfrak {h}}\nolimits} % motive functor
\newcommand{\CHM }{\mathop{\rm {CHM}}\nolimits}% category of Chow motives
\newcommand{\Irr }{\mathop{\rm {Irr}}\nolimits}% Irreducible components/representations
\renewcommand{\1}{\mathop{\mathds{1}}\nolimits} %trivial Chow motive
\renewcommand{\L}{\mathop{\mathds{L}}\nolimits} %Lefschetz motive
\newcommand{\Smproj }{\mathop{\rm {SmProj}}\nolimits}

\begin{document}

\title[Multiplicative McKay correspondence for surfaces]{Motivic Multiplicative McKay correspondence for surfaces}
\author{Lie Fu}
\address{Institut Camille Jordan, Universit\'e Claude Bernard Lyon 1, France}
\email{fu@math.univ-lyon1.fr}

\author{Zhiyu Tian}
\address{CNRS, Institut Fourier, Universit\'e Grenoble Alpes, France}
\email{zhiyu.tian@univ-grenoble-alpes.fr}

\thanks{Lie Fu is supported by the Agence Nationale de la Recherche (ANR) through ECOVA (ANR-15-CE40-0002) and LABEX MILYON (ANR-10-LABX-0070) of Universit\'e de Lyon. Zhiyu Tian is partially supported by the funding ``Accueil des Nouveaux Arrivants" of IDEX, Universit\'e Grenoble Alpes. Lie Fu and Zhiyu Tian are supported by ANR through HodgeFun (ANR-16-CE40-0011), and by \emph{Projet Exploratoire Premier Soutien} (PEPS) Jeunes chercheur-e-s 2016 operated by Insmi and  \emph{Projet Inter-Laboratoire} 2016 and 2017 by F\'ed\'eration de Recherche en Math\'ematiques Rh\^one-Alpes/Auvergne CNRS 3490.}

\begin{abstract}
We revisit the classical two-dimensional McKay correspondence in two respects: The first one, which is the main point of this work, is that we take into account of the multiplicative structure given by the orbifold product; second, instead of using cohomology, we deal with the Chow motives. More precisely, we prove that for any smooth proper two-dimensional orbifold with projective coarse moduli space, there is an isomorphism of algebra objects, in the category of complex Chow motives, between the motive of the minimal resolution and the orbifold motive. In particular, the complex Chow ring (\resp Grothendieck ring, cohomology ring, topological K-theory) of the minimal resolution is isomorphic to the complex orbifold Chow ring (\resp Grothendieck ring, cohomology ring, topological K-theory) of the orbifold surface. This confirms the two-dimensional \emph{Motivic Crepant Resolution Conjecture}. 
\end{abstract}
\maketitle
%\setcounter{tocdepth}{1}
%\tableofcontents
\section{Introduction}\label{sect:intro}

Finite subgroups of $\SL_{2}(\C)$ are classically studied by Klein \cite{MR0080930} and Du Val \cite{MR0169108}. A complete classification (up to conjugacy) is available\,: cyclic, binary dihedral, binary tetrahedral, binary octahedral and binary icosahedral. The last three types correspond to the groups of symmetries of Platonic solids\footnote{\ie regular polyhedrons in $\R^{3}$.} as the names indicate.
Let $G\subset \SL_{2}(\C)$ be such a (non-trivial) finite subgroup acting naturally on the vector space $V\deff \C^{2}$. The quotient $X\deff V/G$ has a unique rational double point\footnote{Such (isolated) surface singularities are also known as Klein, Du Val, Gorenstein, canonical, simple or A-D-E singularities according to different points of view.}. 
Let $f: Y\to X$ be the minimal resolution of singularities:
\begin{displaymath}
\xymatrix{
& V\ar[d]^{\pi}\\
Y\ar[r]^{f} & X
}
\end{displaymath}
which is a crepant resolution, that is, $K_{Y}=f^{*}K_{X}$. The exceptional divisor, denoted by $E$, consists of a union of $(-2)$-curves\footnote{\ie smooth rational curve with self-intersection number equal to $-2$.} meeting transversally.

The classical McKay correspondence (\cite{MR604577}, \cf also \cite{MR1886756}) establishes a bijection between the set $\Irr'(G)$ of non-trivial irreducible representations of $G$ on the one hand and the set $\Irr(E)$ of irreducible components of $E$ on the other hand\,:
\begin{eqnarray*}
\Irr'(G)&\isom& \Irr(E)\\
\rho&\mapsto& E_{\rho}.
\end{eqnarray*}
Thus $E=\bigcup_{\rho\in \Irr'(G)}E_{\rho}$.
Moreover, this bijection respects the `incidence relations'\,: precisely, for any $\rho_{1}\neq\rho_{2}\in \Irr'(G)$, the intersection number $(E_{\rho_{1}}\cdot E_{\rho_{2}})$, which is 0 or 1, is equal to the multiplicity of $\rho_{2}$ in $\rho_{1}\otimes V$ (hence is also equal to the multiplicity of $\rho_{1}$ in $\rho_{2}\otimes V$), where $V$ is the 2-dimensional natural representation via $G\subset \SL(V)$. All these informations can be encoded into Dynkin diagrams of A-D-E type, which is on the one hand the dual graph of the exceptional divisor $E$ and on the other hand the \emph{McKay graph} of the non-trivial irreducible representations of $G$, with respect to the preferred representation $V$. Apart from the original observation of McKay, there are many approaches to construct this correspondence geometrically and to extend it to higher dimensions\,: K-theory of sheaves \cite{MR740077}, $G$-Hilbert schemes \cite{ReidMcKayPreprint}, \cite{MR1420598}, \cite{MR1783852}, \cite{MR1714824}, motivic integration \cite{MR1672108}, \cite{MR1714818}, \cite{MR1664700}, \cite{MR1905024}, \cite{MR2027195}, \cite{MR2069013} and derived categories \cite{MR1824990} \etc We refer the reader to Reid's note of his Bourbaki talk \cite{MR1886756} for more details and history. 

Following Reid \cite{ReidMcKayPreprint}, one can recast the above McKay correspondence (the bijection) as follows: \emph{the isomorphism classes of irreducible representations index a basis of the homology of the resolution $Y$}. This is of course equivalent to say that \emph{the conjugacy classes of $G$ index a basis of the cohomology of $Y$}. The starting point of this paper is that the quotient $X=V/G$ is the coarse moduli space of a smooth orbifold/Deligne--Mumford stack $\rX\deff [V/G]$, and that  the (co)homology of the coarse moduli space $|I\rX|$ of its \emph{inertia stack} $I\rX$ has a basis indexed by the conjugacy classes of $G$. Thus Reid's McKay correspondence can be stated as an isomorphism of vector spaces:
$$H^{*}(Y)\isom H^{*}(|I\rX|).$$
Chen and Ruan defined in \cite{MR2104605} the \emph{orbifold} cohomology and the \emph{orbifold product} (\ie Chen--Ruan cohomology) for any smooth orbifold. See Definition \ref{def:OrbTheo} for a down-to-earth construction in the global quotient case. By definition,  the orbifold cohomology ring $H^{*}_{orb}([V/G])$ has $H^{*}(|I\rX|)$ as the underlying vector space. Therefore it is natural to ask whether there is a multiplicative isomorphism (of algebras)
$$H^{*}(Y)\isom H^{*}_{orb}([V/G]).$$
None of the aforementioned beautiful theories (K-theory, $G$-Hilbert schemes, motivic integration and derived categories) produces an isomorphism which respects the multiplicative structures. Nevertheless, the existence of such an isomorphism of algebras is known.  
For example, it is a baby case of the result of Ginzburg--Kaledin \cite{MR2065506} on symplectic resolutions of symplectic quotient singularities. An explicit isomorphism between the equivariant orbifold quantum cohomology of $[V/G]$ and the equivariant cohomology of its minimal resolution is proposed by Bryan--Graber--Pandharipande in \cite{MR2357679}, which is verified for the $\C^{2}/{\Z_{3}}$ case (see also the related work \cite{MR2377888}, \cite{MR2551767}). We will use the same formula to construct our multiplicative isomorphism. 

This isomorphism fits perfectly into Ruan's following more general Cohomological Crepant Resolution Conjecture (CCRC)\,:

\begin{conj}[CCRC \cite{MR2234886}]\label{conj:CCRC}
Let $M$ be a smooth projective variety endowed with a faithful action of a finite group $G$. Assume that the quotient $X\deff M/G$ is Gorenstein, then for any crepant resolution $Y\to X$, there is an isomorphism of graded $\C$-algebras:
\begin{equation}\label{eqn:CCRC}
H^{*}_{qc}(Y, \C)\isom H^{*}_{orb}\left([M/G], \C\right).
\end{equation}
More generally, given a smooth proper orbifold $\rX$ with underlying singular variety $X$ being Gorenstein, then for any crepant resolution $Y\to X$, we have an isomorphism of graded $\C$-algebras:
$$H^{*}_{qc}(Y, \C)\isom H^{*}_{orb}\left(\rX, \C\right).$$
\end{conj}

Here the left hand side is the \emph{quantum corrected} cohomology algebra, whose underlying graded vector space  is just $H^{*}(Y, \C)$, endowed with the cup product with quantum corrections related to Gromov--Witten invariants with curve classes contracted by the crepant resolution, as defined in \cite{MR2234886}. Since we only consider in this paper the two-dimensional situation, the Gromov--Witten invariants always vanish hence there are no quantum corrections involved. See Lemma  \ref{lemma:qcVanish} for this vanishing. 

Conjecture \ref{conj:CCRC} suggests that one should consider the existence of such multiplicative McKay correspondence in the global situation (instead of a quotient of a vector space by a finite group), that is, a Gorenstein quotient of a surface by a finite group action, or even more generally a two-dimensional proper Gorenstein orbifold.
Our following main result confirms this, which also pushes the (surface) McKay correspondence to the motivic level:

\begin{thm}[Motivic multiplicative global McKay correspondence]\label{thm:main}
Let $\rX$ be a smooth proper two-dimensional Deligne--Mumford stack with isolated stacky points. Assume that $\rX$ has projective coarse moduli space $X$ with Gorenstein singularities. Let $Y\to X$ be the minimal resolution. Then we have an isomorphism of algebra objects in the category $\CHM_{\C}$ of Chow motives with complex coefficients:
\begin{equation}\label{eqn:main}
\h(Y)_{\C}\isom \h_{orb}\left(\rX\right)_{\C}.
\end{equation}
In particular, one has an isomorphism of $\C$-algebras:
\begin{eqnarray*}
\CH^{*}(Y)_{\C} &\isom& \CH^{*}_{orb}\left(\rX\right)_{\C};\\
H^{*}(Y, \C) &\isom& H^{*}_{orb}\left(\rX, \C\right);\\
K_{0}(Y)_{\C} &\isom& K_{orb}\left(\rX\right)_{\C};\\
K^{top}(Y)_{\C} &\isom& K^{top}_{orb}\left(\rX\right)_{\C}.
\end{eqnarray*}
\end{thm}

This result also confirms the 2-dimensional case of the so-called \emph{Motivic HyperK\"ahler Resolution Conjecture} studied in \cite{MHRCKummer} and \cite{MHRCK3}. See \S \ref{subsect:Motives} for the basics of Chow motives. 

As the definitions of the orbifold theories are particularly explicit and elementary for the global quotient stacks (\cf \S\ref{subsect:OrbQuo}), we deliberately treat the global quotient case (\S\ref{sect:ProofQuotient}) and the general case (\S\ref{sect:ProofGeneral}) separately.

\noindent\textbf{Convention :} All Chow rings and K-theories are with rational coefficients unless otherwise stated. $\CHM$ is the category of Chow motives with rational coefficients and $\h: \Smproj^{op}\to \CHM$ is the (contra-variant) functor that associates a smooth projective variety its Chow motive (\S \ref{subsect:Motives}). An \emph{orbifold} means a separated Deligne--Mumford stack of finite type with trivial stabilizer at the generic point. We work over an algebraically closed field of characteristic zero. 

\noindent\textbf{Acknowledgement :} The authors want to thank Samuel Boissi\`ere, C\'edric Bonnaf\'e, Philippe Caldero, J\'er\^ome Germoni and Dmitry Kaledin for interesting discussions and also the referee for his or her very helpful suggestions. The most part of the paper is prepared when Lie Fu is staying with his family at the Hausdorff Institute of Mathematics for the 2017 trimester program on K-theory. He thanks Bonn University for providing the perfect working condition in HIM and the relaxing living style in such a beautiful city.

\section{Crepant resolution conjecture}
Let us give the construction of the \emph{orbifold Chow motive} (as an algebra object) and the \emph{orbifold Chow ring}. we will first give the down-to-earth definition for an orbifold which is a global Gorenstein quotient by a finite group\,; then we invoke the techniques in \cite{MR2450211} to give the construction in the general setting of Deligne--Mumford stacks.  We refer to our previous work \cite{MHRCKummer} (joint with Charles Vial), \cite{MHRCK3} as well as the original sources (for cohomology and Chow rings) \cite{MR2104605}, \cite{MR1971293}, \cite{MR2450211}, \cite{MR2285746} for the history and more details. For the convenience of the reader, we start with a reminder on the basic notions of Chow motives. 

\subsection{The category of Chow motives}\label{subsect:Motives}
The idea of (pure) \emph{motives}, proposed initially by Grothendieck, is to construct a universal cohomology theory $X\mapsto \h(X)$ for smooth projective varieties. His construction uses directly the algebraic cycles on the varieties together with some natural categorical operations. On the one hand, motives behave just like the classically considered Weil cohomology theories\,; on the other hand, they no longer take values in the category of vector spaces but in some additive idempotent-complete rigid symmetric mono\"idal category. Although the construction works for any adequate equivalence relation on algebraic cycles, we use throughout this paper the finest one, namely the \emph{rational} equivalence, so that our results will hold for Chow groups and imply the other analogous ones, on cohomology for instance, by applying appropriate realization functors. 

Fix a base field $k$. The category of \emph{Chow motives} over the field $k$ with rational coefficients, denoted by $\CHM$, is defined as follows (\cf \cite{MR2115000} for a more detailed treatment). An object, called a \emph{Chow motive}, is a triple $(X, p, n)$, where $n$ is an integer, $X$ is a smooth projective variety over $k$ and $p\in\CH^{\dim X}(X\times X)$ is a projector, that is, $p\circ p=p$ as correspondences. Morphisms between two objects $(X, p, n)$ to $(Y, q, m)$ form the following $\Q$-vector subspace of $\CH^{m-n+\dim X}(X\times Y)$\,:
$$\Hom_{\CHM}\left((X, p, n), (Y, q, m)\right):=q\circ\CH^{m-n+\dim X}(X\times Y)\circ p.$$
The composition of morphisms are defined by the composition of correspondences.
We have the following naturally defined contra-variant functor from the category of smooth projective varieties to the category of Chow motives\,:
\begin{eqnarray*}
 \h: \Smproj^{op}&\to& \CHM\\
 X&\mapsto& (X,\Delta_{X}, 0)\\
(f\colon X\to Y) &\mapsto & {}^{t}\Gamma_{f}\in \CH^{\dim X}(Y\times X)
\end{eqnarray*}
where ${}^{t}\Gamma_{f}$ is the transpose of the graph of $f$. The image $\h(X)$ is called the Chow motive of $X$.

The category $\CHM$ is additive with direct sum induced by the disjoint union of varieties. By construction, $\CHM$ is idempotent complete (\ie pseudo-abelian)\,: for any motive $M$ and any projector of it, that is, $\phi\in \End_{\CHM}(M)$ such that $\phi\circ \phi=\phi$, we have $M\cong \im(\phi)\oplus \im(\id_{M}-\phi)$. As an example, let us recall the definition of the so-called \emph{reduced motive}: for a smooth projective variety $X$ together with a chosen $k$-rational point $x$, the composition of $\Spec k\xrightarrow{x} X$ with the structure morphism $X\to \Spec k$ is identity, hence defines a projector $\phi$ of $\h(X)$. The reduced motive of the pointed variety $(X, x)$, denoted by $\tilde{\h}(X)$, is by definition $\im(\id_{\h(X)}-\phi)$. One can show that the isomorphism class of $\tilde{\h}(X)$ is independent of the choice of the point $x$ (\cf \cite[Exemple 4.1.2.1]{MR2115000}).

There is also a natural symmetric mono\"idal structure on $\CHM$, compatible with the additive structure, given by 
$$(X, p, n)\otimes (Y, q, m):=(X\times Y, p\times q, n+m).$$ Hence the \emph{K\"unneth formula}: $\h(X\times Y)\cong \h(X)\otimes \h(Y)$ holds for any smooth projective varieties $X$ and $Y$. 
Moreover, this tensor category is \emph{rigid}, with the following duality functor
$$(X, p, n)^{\vee}:=(X, {}^{t}p, \dim X-n).$$ Given an integer $n$, the motive $(\Spec k, \Delta_{\Spec k}, n)$ is called the $n$-th \emph{Tate motive} and is denoted by $\1(n)$. They are the tensor invertible objects. For any motive $M$, the tensor product $M\otimes \1(n)$ is denoted by $M(n)$ and called the $n$-th \emph{Tate twist} of $M$. In particular, we have the \emph{Poicar\'e duality}: $\h(X)^{\vee}\cong \h(X)(\dim X)$ for any smooth projective variety $X$.
The \emph{Lefschetz motive} is $\L:=\1(-1)$. 
One checks that the reduced motive of $\P^1$ is isomorphic to $\L$.

The functor $\h$ is considered as a cohomology theory and it is universal in the sense that any Weil cohomology theory must factorize through $\h$. We can extend the notion of Chow groups from varieties to motives by defining for any integer $i$ and any Chow motive $M$, $$\CH^{i}(M):=\Hom_{\CHM}\left(\1(-i), M\right),$$ Hence the Chow groups of a smooth projective variety $X$ is recovered as $\CH^{i}(X)=\CH^{i}\left(\h(X)\right).$

In all the above constructions, one can replace for the coefficient field the rational numbers by the complex numbers and obtain the category of complex Chow motives $\CHM_{\C}$.
\subsection{Orbifold theory : global quotient case}\label{subsect:OrbQuo}

 Let $M$ be a smooth projective variety and $G$ be a finite group acting faithfully on $M$. Assume that $G$ preserves locally the canonical bundle\,: for any $x\in M$ fixed by $g\in G$, the differential $Dg\in \SL(T_{x}M)$. This amounts to require that the quotient $X\deff M/G$ has only Gorenstein singularities. Denote by $M^{g}=\left\{x\in M~\mid~ gx=x\right\}$ the fixed locus of $g\in G$, $M^{\langle g,h \rangle}=M^{g}\cap M^{h}$ (with the reduced scheme structure) and $\rX\deff [M/G]$ the quotient smooth Deligne--Mumford stack.
 
\begin{defi}[Orbifold theories]\label{def:OrbTheo}
We define an auxiliary algebra object $\h(M, G)$ in $\CHM$ with $G$-action, and the orbifold motive $\h([M/G])$ will be its subalgebra of invariants. The definitions for Chow rings, cohomology and K-theory are similar.
\begin{enumerate}[$(1^{\circ})$]
\item For any $g\in G$, the \emph{age function}, denoted by $\age(g)$, is a $\Z$-valued locally constant function on $M^{g}$, whose value on a connected component $Z$ is $$\age(g)|_{Z}\deff \sum_{j=0}^{r-1} \frac{j}{r} \rank(W_{j}),$$
where $r$ is the order of $g$, $W_{j}$ is the eigen-subbundle of the restricted tangent bundle $TM|_{Z}$, for the natural automorphism induced by $g$, with eigenvalue $e^{\frac{2\pi i}{r}j}$. The age function is invariant under conjugacy.
\item We endow the direct sums $$\h(M, G)\deff \bigoplus_{g\in G} \h(M^{g})(-\age(g))$$
$$\CH^{*}(M, G)\deff \bigoplus_{g\in G} \CH^{*-\age(g)}(M^{g})$$ 
$$H^{*}(M, G)\deff \bigoplus_{g\in G} H^{*-2\age(g)}(M^{g})$$
$$K(M, G)\deff \bigoplus_{g\in G} K_{0}(M^{g})_{\Q}$$
$$K^{top}(M, G)\deff \bigoplus_{g\in G} K^{top}(M^{g})_{\Q}$$
with the natural $G$-action induced by the following action: for any $g, h\in G$,
\begin{eqnarray*}
h: M^{g}&\lra{\isom}& M^{hgh^{-1}}\\
x&\mapsto& hx.
\end{eqnarray*}
\item For any $g\in G$, define 
\begin{equation*}
V_{g}\deff\sum_{j=0}^{r-1}\frac{j}{r}[W_{j}]\in K_{0}(M^{g})_{\Q},
\end{equation*}
whose virtual rank is $\age(g)$, where $r$ and $W_{j}$ 's are as in $(1^{\circ})$.\\
\item For any $g_{1}, g_{2}\in G$, let $g_{3}=g_{2}^{-1}g_{1}^{-1}$, we define the (virtual class of )  the \emph{obstruction bundle} on the fixed locus $M^{\langle g_{1}, g_{2} \rangle}$ by 
\begin{equation}\label{eqn:ObsBun}
F_{g_{1}, g_{2}}\deff \left.V_{g_{1}}\right\vert_{M^{<g_{1}, g_{2}>}}+ \left.V_{g_{2}}\right\vert_{M^{<g_{1}, g_{2}>}}+ \left.V_{g_{3}}\right\vert_{M^{<g_{1}, g_{2}>}}+TM^{<g_{1},g_{2}>}-\left.TM\right\vert_{M^{<g_{1},g_{2}>}}  \in K_{0}\left(M^{<g_{1}, g_{2}>}\right)_{\Q}.
\end{equation}
\item The \emph{orbifold product} $\star_{orb}$ is defined as follows: given $g, h\in G$, let $\iota: M^{<g,h>}\inj M$ be the natural inclusion. 
\begin{itemize}
\item For cohomology: 
\begin{eqnarray*}
\star_{orb}: H^{i-2\age(g)}(M^{g}) \times H^{j-2\age(h)}(M^{h}) &\to& H^{i+j-2\age(gh)}(M^{gh})\\ 
(\alpha, \beta)&\mapsto&  \iota_{*}\left(\alpha|_{M^{<g,h>}}\smile\beta|_{M^{<g,h>}}\smile c_{top}(F_{g,h})\right)
\end{eqnarray*}
\item For Chow groups: 
\begin{eqnarray*}
\star_{orb}: \CH^{i-\age(g)}(M^{g}) \times \CH^{j-\age(h)}(M^{h}) &\to& \CH^{i+j-\age(gh)}(M^{gh})\\ 
(\alpha, \beta)&\mapsto&  \iota_{*}\left(\alpha|_{M^{<g,h>}}\cdot\beta|_{M^{<g,h>}}\cdot c_{top}(F_{g,h})\right)
\end{eqnarray*}
\item For K-theory: 
\begin{eqnarray*}
\star_{orb}: K_{0}(M^{g})_{\Q} \times K_{0}(M^{h})_{\Q} &\to& K_{0}(M^{gh})_{\Q}\\ 
(\alpha, \beta)&\mapsto&  \iota_{!}\left(\alpha|_{M^{<g,h>}}\cdot\beta|_{M^{<g,h>}}\cdot \lambda_{-1}(F_{g,h}\dual)\right)
\end{eqnarray*}
where $\lambda_{-1}$ is obtained from the Lambda operation $\lambda_{t}\colon K_{0}(M^{<g,h>})\to K_{0}(M^{<g,h>})\llbracket t\rrbracket$ by evaluating  $t=-1$ (\cf \cite[Chapter II, \S 4]{MR3076731}).
The definition for topological K-theory is similar.
\item For motives:  $\star_{orb}: \h(M^{g})(-\age(g))\otimes \h(M^{h})(-\age(h))\to \h(M^{gh})(-\age(gh))$ is determined by the correspondence
\begin{equation*}
\delta_{*}(c_{top}(F_{g,h}))\in \CH^{\dim M^{g}+\dim M^{h}+\age(g)+\age(h)-\age(gh)}(M^{g}\times M^{h}\times M^{gh}),
\end{equation*}
where $\delta: M^{<g, h>}\to M^{g}\times M^{h}\times M^{gh}$ is the natural morphism sending $x$ to $(x,x,x)$.
\end{itemize}

\item Finally, we take the subalgebra of invariants whose existence is guaranteed by the idempotent completeness of $\CHM$ (see \S\ref{subsect:Motives}) :
$$\h_{orb}\left([M/G]\right)\deff \h\left(M, G\right)^{G};$$
$$\CH^{*}_{orb}\left([M/G]\right)\deff \left(\CH^{*}(M, G), \star_{orb}\right)^{G};$$
and similarly $$H^{*}_{orb}\left([M/G]\right)\deff \left(H^{*}(M, G), \star_{orb}\right)^{G};$$
$$K_{orb}\left([M/G]\right)\deff \left(K(M, G), \star_{orb}\right)^{G};$$
$$K^{top}_{orb}\left([M/G]\right)\deff \left(K^{top}(M, G), \star_{orb}\right)^{G}.$$
These are commutative $\Q$-algebras and depend only on the stack $[M/G]$ (not the presentation).
\end{enumerate}
\end{defi}

\subsection{Orbifold theory\,: general case}\label{subsect:GeneralOrbTheory}
Let $\rX$ be a smooth proper orbifold with projective coarse moduli space $X$ with Gorenstein singularities. Recall that under the Gorenstein assumption, the age function takes values in integers.
Define the \emph{orbifold Chow motive} and \emph{orbifold Chow group} as follows:
$$\h_{orb}(\rX):=\h(I\rX)(-\age):=\oplus_{i}\h(I\rX_{i})(-\age_{i}),$$ 
$$\CH^{*}_{orb}(\rX):= \CH^{*-\age}(I\rX):=\oplus_{i}\CH^{*-\age_{i}}(I\rX_{i})\,;$$
where the theory of Chow ring (with rational coefficients) as well as the intersection theory of a stack is the one developed by Vistoli in \cite{MR1005008}\,; the theory of Chow motives for smooth proper Deligne--Mumford stacks is the so-called \emph{DMC motives}\footnote{DMC stands for Deligne--Mumford Chow.} developed by Behrend--Manin in \cite{MR1412436} and reviewed in To\"en \cite[\S 2. First construction]{MR1784411}, which is proven in \cite[Theorem 2.1]{MR1784411} to be equivalent to the usual category of Chow motives\,; $I\rX=\coprod_{i} I\rX_{i}$ is the decomposition into connected components while the \emph{age function} $\age$ is the locally constant function whose value on $I\rX_{i}$ is $\age_{i}$ which is Chen--Ruan's \emph{degree shifting number} defined in \cite[\S 3.2]{MR2104605}. Let us also point out that To\"en's second construction in  \cite[\S 3]{MR1784411} of Chow motives of Delign--Mumford stacks is very close to the orbifold Chow motive defined above with the only difference being the age-shifting.

Now the key point is to put a product structure on $\h_{orb}(\rX)$ and $\CH^{*}_{orb}(\rX)$. 
 Consider the moduli space $K_{0,3}(\rX, 0)$, constructed by Abramovich--Vistoli \cite{MR1862797}, of 3-pointed \emph{twisted} stable maps of genus zero with trivial curve class. It comes equipped with a virtual fundamental class $[K_{0,3}(\rX, 0)]^{vir}\in \CH_{\dim X}\left(K_{0,3}(\rX, 0)\right)$ together with three (proper) evaluation maps: $e_{i}: K_{0,3}(\rX, 0)\to I\rX$ with target being the inertia stack (\cite{MR2450211}). Note that in general, the evaluation morphism has target in a different stack, the rigidified cyclotomic inertial stack (\cite[Section 3.4]{MR2450211}). However, in the smooth orbifold case, one can prove that the evaluation morphisms of the degree $0$ twisted stable maps land in the inertial stack \cite[Section 1.3.1]{EJKDuke}.
 
  Pushing forward the virtual fundamental class gives the class
$$\gamma:=(e_{1}, e_{2}, \check e_{3})_{*}\left([K_{0,3}(\rX, 0)]^{vir}\right)\in \CH_{\dim X}(I\rX^{3}),$$ where $\check e_{3}$ is the composition of the evaluation map $e_{3}$ and the involution $I\rX\to I\rX$ inverting the group element (\cf \cite{MR2450211}); and we are using again Vistoli's Chow groups (\cite{MR1005008}). The \emph{orbifold product} for the orbifold Chow ring is defined as the action of the correspondence $\gamma$:
\begin{eqnarray*}
\CH^{*}_{orb}(\rX) \times & \CH^{*}_{orb}(\rX)\to & \CH^{*}_{orb}(\rX)\\
\parallel &&\parallel\\
\CH^{*-\age}(I\rX)\times& \CH^{*-\age}(I\rX)  \to &\CH^{*-\age}(I\rX)\\
(\alpha,\beta)&\mapsto&  \pr_{3,*}\left(\pr_{1}^{*}(\alpha)\cdot \pr_{2}^{*}(\beta)\cdot \gamma\right)
\end{eqnarray*}
It can be checked (\cf \cite[Theorem 7.4.1]{MR2450211}) that the age shifting makes the above orbifold product  additive with respect to the degrees (otherwise, it is not!).
Similarly, we can define the multiplicative structure on $\h_{orb}(\rX)$ to be
\begin{eqnarray*}
\gamma\in \CH_{\dim X}(I\rX^{3})&=&\Hom_{\CHM}\left(\h(I\rX)(-\age)\otimes \h(I\rX)(-\age), \h(I\rX)(-\age)\right)\\
&=&\Hom_{\CHM}\left(\h_{orb}(\rX)\otimes \h_{orb}(\rX), \h_{orb}(\rX)\right).
\end{eqnarray*}

Thanks to \cite[Theorem 7.4.1]{MR2450211}, this product structure is associative. On the other hand, when $\rX$ is a finite group global quotient stack, the main result of \cite[\S 8]{MR2285746} implies that the elementary construction in \S \ref{subsect:OrbQuo} actually recovers the above abstract construction.

\subsection{Crepant resolution conjectures}

With orbifold theories being defined, we can speculate that a motivic or K-theoretic version of the Crepant Resolution Conjecture \ref{conj:CCRC} should hold. But the problem is that in the definition of the \emph{quantum corrections}, there is the subtle convergence property which is difficult to make sense in general for Chow groups / motives or for K-theory. Therefore, we will look at some cases that these quantum corrections actually vanish \apriori\,:

\paragraph{\textbf{Case 1: Hyper-K\"ahler resolution}}

 The first one is when the resolution $Y$ is holomorphic symplectic, which implies that all (Chow-theoretic, K-theoretic or cohomological) Gromov--Witten invariants vanish (see the proof of \cite[Lemma 8.1]{MHRCK3}). In this case, we indeed have the following \emph{Motivic Hyper-K\"ahler Resolution Conjecture} (MHRC), proposed in \cite{MHRCKummer}:
\begin{conj}[MHRC \cite{MHRCKummer}, \cite{MHRCK3}]\label{conj:MHRC}
Let $M$ be a smooth projective holomorphic symplectic variety endowed with a faithful symplectic action of a finite group $G$. If the quotient $X\deff M/G$ has a crepant resolution $Y\to X$, then there is an isomorphism of algebra object in the category $\CHM_{\C}$ of complex Chow motives:
$$\h(Y)\isom \h_{orb}([M/G]).$$
In particular, we have an isomorphism of graded $\C$-algebras:
\begin{equation*}
\CH^{*}(Y)_{\C}\isom \CH^{*}_{orb}\left([M/G]\right)_{\C}.
\end{equation*}
\end{conj}
Thanks to the orbifold Chern character isomorphism constructed by Jarvis--Kaufmann--Kimura in \cite{MR2285746}, MHRC also implies the K-theoretic Hyper-K\"ahler Resolution Conjecture of \loccit.
Conjecture \ref{conj:MHRC} is proven in our joint work with Charles Vial \cite{MHRCKummer} for Hilbert schemes of abelian varieties and generalized Kummer varieties and in \cite{MHRCK3} for Hilbert schemes of K3 surfaces.

\paragraph{\textbf{Case 2: Surface minimal resolution}}
The second one is the main purpose of the article, namely the surface case, \ie $\dim(Y)=2$. In this case, the vanishing of quantum corrections is explained in the following lemma.

\begin{lemma}\label{lemma:qcVanish}
Let $X$ be a surface with Du Val singularities and $\pi: Y\to X$ be the minimal resolution. Then the virtual fundamental class of $\overline{M_{0,3}}\left(Y,\beta\right)$ is rationally equivalent to zero for any curve class $\beta$ which is contracted by $\pi$.
\end{lemma}
\begin{proof}
Consider the forgetful-stabilization morphism $$f: \overline{M_{0,3}}\left(Y,\beta\right)\to \overline{M_{0,0}}\left(Y,\beta\right).$$ By the general theory, the virtual fundamental class of $\overline{M_{0,3}}\left(Y,\beta\right)$ is the pull-back of the virtual fundamental class of $\overline{M_{0,0}}\left(Y,\beta\right)$. However, the virtual dimension of $\overline{M_{0,0}}\left(Y,\beta\right)$ is $(\beta\cdot K_{Y})+(\dim Y-3)=-1$ since $\pi$ is crepant. Therefore, both moduli spaces have zero virtual fundamental class in Chow group, cohomology or K-theory. 
\end{proof}

Thanks to the vanishing of quantum corrections, the motivic version of the Crepant Resolution Conjecture \ref{conj:CCRC} for surfaces is exactly the content of our main Theorem \ref{thm:main}. See the precise statement in Introduction. We will first give the proof for stacks which are finite group global quotients in \S \ref{sect:ProofQuotient}, then the proof in the general case in \S \ref{sect:ProofGeneral}.

\section{Proof of Theorem \ref{thm:main} : global quotient case}\label{sect:ProofQuotient}
In this section, we show Theorem \ref{thm:main} in the following setting: $S$ is a smooth projective surface, $G$ is a finite group acting faithfully on $S$ such that the canonical bundle is locally preserved (Gorenstein condition), $X\deff S/G$ is the quotient surface (with Du Val singularities) and $Y\to X$ is the minimal (crepant) resolution. Recall that $\L\deff \1(-1)$ is the Lefschetz motive in $\CHM$ (\S \ref{subsect:Motives}).

For any $x\in S$, let $$G_{x}\deff \{g\in G ~\vert~ gx=x\}$$ be the stabilizer. Let $\Irr(G_{x})$ be the set of isomorphism classes of irreducible representations of $G_{x}$ and $\Irr'(G_{x})$ be that of non-trivial ones. We remark that by assumption, there are only finitely many points of $S$ with non-trivial stabilizer. 

\subsection{Resolution side}
We first compute the Chow motive algebra (or Chow ring) of the minimal resolution $Y$. 

For any $x\in S$, we denote by $\bar x$ its image in $S/G$.
The Chow motive of $Y$ has the following decomposition in $\CHM$\,:
\begin{equation}\label{eqn:decompRes}
\h\left(Y\right)\isom \h(S)^{G}\oplus\bigoplus_{\bar x\in S/G}\bigoplus_{\rho\in \Irr'(G_{x})}\L_{\bar x, \rho}\isom \left(\h(S)\oplus\bigoplus_{x\in S}\bigoplus_{\rho\in \Irr'(G_{x})}\L_{x, \rho}\right)^{G},
\end{equation}
 where $\L_{\bar x, \rho}$ and $\L_{x,\rho}$ are both the Lefschetz motive $\L$ corresponding to the irreducible component of the exceptional divisor over $\bar x$, indexed by the non-trivial irreducible representation $\rho$ of $G_{x}$ via the classical McKay correspondence. The second isomorphism in \eqref{eqn:decompRes} being just a trick of reindexing, let us explain a bit more on the first one. Let $f:Y\to S/G$ be the minimal resolution of singularities. By the classical McKay correspondence, over each singular point $\bar x\in S/G$, the exceptional divisor $E_{\bar x}:=f^{-1}(\bar x)$ is a union (with A-D-E configuration) of smooth rational curves $\cup_{\rho\in \Irr'(G_{x})}E_{\bar{x},\rho}$. As $f$ is obviously a semi-small morphism, we can invoke the motivic decomposition of De Cataldo--Migliorini \cite[Theorem 1.0.1]{MR2067464}, with the stratification being $S/G=(S/G)_{reg}\cup \Sing(S/G)$, to obtain directly the first isomorphism in \eqref{eqn:decompRes}. It is then not hard to follow the proof in \loccit to see that the first isomorphism in  \eqref{eqn:decompRes} is induced by the pull-back $f^{*}={}^{t}\Gamma_{f}:\h(S)^{G}=\h(S/G)\to \h(Y)$ together with the push-forward along the inclusions $\L=\tilde{ \h}(E_{\bar{x}, \rho})\xrightarrow{i_{\bar x, \rho,*}} \h(Y)$, where $\tilde{ \h}$ is the reduced motive (see \S\ref{subsect:Motives}). We remark that the inverse of the isomorphism \eqref{eqn:decompRes} is more complicated to describe and involves the inverse of the intersection matrix (\cf the definition of $\Gamma'$ in the end of \cite[\S 2]{MR2067464}).

The product structure of $\h(Y)$ is determined as follows via the above decomposition \eqref{eqn:decompRes}, which also uses the classical McKay correspondence. Let $i_{x}:\{x\}\inj S$ be the natural inclusion. 
\begin{itemize}
\item $\h(S)\otimes \h(S)\lra{\delta_{S}} \h(S)$ is the usual product induced by the small diagonal of $S^{3}$.
\item For any $x$ with non-trivial stabilizer $G_{x}$ and any $\rho\in\Irr'(G_{x})$, $$\h(S)\otimes \L_{x,\rho}\lra{i_{x}^{*}} \L_{x,\rho} $$  is determined by the class $x\in \CH^{2}(S)=\Hom(\h(S)\otimes \L, \L)$.
\item For any $\rho\in \Irr'(G_{x})$ as above, $$\L_{x,\rho}\otimes \L_{x,\rho}\lra{-2i_{x, *}} \h(S),$$ is determined by $-2x\in\CH^{2}(S)$. The reason is that each component of the exceptional divisor is a smooth rational curve of self-intersection number equal to $-2$.
\item For any $\rho_{1}\neq \rho_{2}\in \Irr'(G_{x})$, 
\begin{itemize}
\item If they are \emph{adjacent}, that is, $\rho_{1}$ appears (with multiplicity 1) in the $G_{x}$-module $\rho_{2}\otimes T_{x}S$, then by the classical McKay correspondence, the components in the exceptional divisor over $\bar x$ indexed by $\rho_{1}$ and by $\rho_{2}$ intersect transversally at one point. Therefore $$\L_{x,\rho_{1}}\otimes \L_{x,\rho_{2}}\lra{i_{x, *}} \h(S),$$ is determined by $x\in\CH^{2}(S)$.
\item If they are not adjacent, then again the classical McKay correspondence tells us that the two components indexed by $\rho_{1}$ and $\rho_{2}$ of the exceptional divisor do not intersect\,; hence $\L_{x,\rho_{1}}\otimes \L_{x,\rho_{2}}\lra{0} \h(S)$ is the zero map.
\end{itemize}
\item The other multiplication maps are zero.
\end{itemize}

The $G$-action on (\ref{eqn:decompRes}) is as follows: 
\begin{itemize}
\item The $G$-action of $\h(S)$ is induced by the original action on $S$.
\item For any $h\in G$, it maps for any $x\in S$ and $\rho\in \Irr'(G_{x})$, the Lefschetz motive $\L_{x, \rho}$ isomorphically to $\L_{hx, h\rho}$, where $h\rho \in\Irr'(G_{hx})$ is the representation which makes the following diagram commutes:
\begin{equation}\label{diag:ActRep}
\xymatrix{
G_{x} \ar[rr]_{\isom}^{g\mapsto hgh^{-1}}\ar[dr]_{\rho} && G_{hx} \ar[dl]^{h\rho}\\
&V_{\rho}&
}
\end{equation}

\end{itemize}

\subsection{Orbifold side}
Now we compute the orbifold Chow motive algebra of the quotient stack $[S/G]$. The computation is quite straight-forward. Here $\L\deff \1(-1)$ is the Lefschetz motive. 

First of all, it is easy to see that $\age(g)=1$ for any element $g\neq \id$ of $G$, and $\age(\id)=0$. By Definition \ref{def:OrbTheo},
\begin{equation}\label{eqn:OrbMotDec}
\h(S, G)=\h(S)\oplus\bigoplus_{\substack{g\in G\\g\neq \id}}\bigoplus_{x\in S^{g}}\L_{x,g}=\h(S)\oplus\bigoplus_{x\in S}\bigoplus_{\substack{g\in G_{x}\\g\neq \id}}\L_{x,g},
\end{equation}
where $\L_{x,g}$ is the Lefschetz motive $\1(-1)$ indexed by the fixed point $x$ of $g$.
\begin{lemma}[Obstruction class]\label{lemma:ObsClass}
For any $g, h\in G$ different from $\id$, the obstruction class is
 $$c_{g,h}={\begin{cases}1 ~~~\text{ if }~~~ g=h^{-1}\\ 0 ~~~\text{ if }~~~ g\neq h^{-1}\end{cases}}$$
\end{lemma}
\begin{proof}
For any $g\neq \id$ and any $x\in S^{g}$, the action of $g$ on $T_{x}S$ is diagonalizable with a pair of conjugate eigenvalues, therefore $V_{g}$ in Definition \ref{def:OrbTheo} is a trivial vector bundle of rank one on $S^{g}$. Hence for any $g, h\in G$ different from $\id$ and $x\in S$ fixed by $g$ and $h$, the dimension of the fiber of the obstruction bundle $F_{g,h}$ at $x$ is
$$\dim F_{g,h}(x)=\dim V_{g}(x)+\dim V_{h}(x)+\dim V_{{(gh)}^{-1}}(x)-\dim T_{x}S,$$
which is 1 if $g\neq h^{-1}$ and is 0 if $g=h^{-1}$. The computation of $c_{g,h}$ follows.
\end{proof}

Once the obstruction classes are computed, we can write down explicitly the orbifold product from Definition \ref{def:OrbTheo}, which is summarized in the following proposition. 
\begin{prop}\label{prop:OrbProd}
The orbifold product on $\h(S, G)$ is given as follows via the decomposition (\ref{eqn:OrbMotDec}):
\begin{eqnarray*}
\h(S)\otimes \h(S)&\lra{\delta_{S}}& \h(S);\\
\h(S)\otimes \L_{x,g}&\lra{i_{x}^{*}}& \L_{x,g} ~~\forall gx=x;\\
\L_{x,g}\otimes \L_{x,g^{-1}}&\lra{i_{x, *}}& \h(S).\\ 
\end{eqnarray*}
where the first morphism is the usual product given by small diagonal; the second and the third morphisms are given by the class $x\in \CH^{2}(S)$ and $i_{x}: \{x\}\inj S$ is the natural inclusion; all the other possible maps are zero. 
\end{prop}

The $G$-action on (\ref{eqn:OrbMotDec}) is as follows by Definition \ref{def:OrbTheo}:
\begin{itemize}
\item The $G$-action on $\h(S)$ is the original action.
\item For any $h\in G$, it maps for any $x\in S$ and $g\neq \id \in G_{x}$, the Lefschetz motive $\L_{x,g}$ isomorphically to $\L_{hx, hgh^{-1}}$.
\end{itemize}

\subsection{The multiplicative correspondence}
With both sides of the correspondence computed, we can give the \emph{multiplicative McKay correspondence} morphism, which is in the category $\CHM_{\C}$ of complex Chow motives. Consider the morphism

\begin{equation}\label{eqn:MKCorr}
\Phi: \h(S)\oplus\bigoplus_{x\in S}\bigoplus_{\rho\in \Irr'(G_{x})}\L_{x, \rho} \to  \h(S)\oplus\bigoplus_{x\in S}\bigoplus_{\substack{g\in G_{x}\\g\neq \id}}\L_{x,g},
\end{equation}
which is given by the following `block diagonal matrix':
\begin{itemize}
\item $\id: \h(S)\to \h(S)$;
\item For each $x\in S$ (with non-trivial stabilizer $G_{x}$), the morphism $$\bigoplus_{\rho\in \Irr'(G_{x})}\L_{x, \rho}\to \bigoplus_{\substack{g\in G_{x}\\g\neq \id}}\L_{x,g}$$ is the `matrix' with coefficient 
$\frac{1}{\sqrt{|G_{x}|}}\sqrt{\chi_{\rho_{0}}(g)-2}\cdot\chi_{\rho}(g)$ at place $(\rho, g)\in \Irr'(G_{x})\times (G_{x}\backslash \{\id\})$,  where $\chi$ denotes the character, $\rho_{0}$ is the natural 2-dimensional representation $T_{x}S$ of $G_{x}$. Note that $\rho_{0}(g)$ has determinant 1, hence its trace $\chi_{\rho_{0}}(g)$ is a real number.
\item The other morphisms are zero.
\end{itemize}

To conclude the main theorem, one has to show three things: $(i)$ $\Phi$ is compatible with the $G$-action;  $(ii)$ $\Phi$ is multiplicative and $(iii)$ $\Phi$ induces an isomorphism $\Phi^{G}$ of complex Chow motives on $G$-invariants. 

\begin{lemma}\label{lemma:Equiv}
$\Phi$ is $G$-equivariant.
\end{lemma}
\begin{proof}
The $G$-action on the first direct summand $\h(S)$ is by definition the same, hence is preserved by $\Phi|_{\h(S)}=\id$. For the other direct summands, since it is a matrix computation, we can treat the Lefschetz motives as 1-dimensional vector spaces: let $E_{x,\rho}$ be the `generator' of $\L_{x,\rho}$ and $e_{x, g}$ be the `generator' of $\L_{x,g}$. 
Then the $G$-actions computed in the previous subsections say that for any $x$ and any $h\in G_{x}$,
$$h. E_{x,\rho}=E_{hx, h\rho}~~~~~~\text{ and }~~~~~~~~h.e_{x,g}=e_{hx, hgh^{-1}},$$
where $h\rho$ is defined in (\ref{diag:ActRep}).

Therefore 
\begin{eqnarray*}
&&\Phi(h.E_{x, \rho})\\
&=&\Phi(E_{hx, h\rho})\\
&=&\frac{1}{\sqrt{|G_{hx}|}}\sum_{g\in G_{hx}}\sqrt{\chi_{\rho_{0}}(g)-2}\,\chi_{h\rho}(g)\, e_{hx,g}\\
&=& \frac{1}{\sqrt{|G_{x}|}}\sum_{g\in G_{x}}\sqrt{\chi_{\rho_{0}}(g)-2}\,\chi_{h\rho}(hgh^{-1})\, e_{hx,hgh^{-1}}\\
&=& \frac{1}{\sqrt{|G_{x}|}}\sum_{g\in G_{x}}\sqrt{\chi_{\rho_{0}}(g)-2}\,\chi_{\rho}(g)\, e_{hx,hgh^{-1}}\\
&=&\frac{1}{\sqrt{|G_{x}|}}\sum_{g\in G_{x}}\sqrt{\chi_{\rho_{0}}(g)-2}\,\chi_{\rho}(g)\, h.e_{x,g}\\
&=&h.\Phi(E_{x, \rho}),
\end{eqnarray*}
where the third equality is a change of variable: replace $g$ by $hgh^{-1}$, the fourth equality follows from the definition of $h\rho$ in (\ref{diag:ActRep})
\end{proof}

\begin{prop}[Multiplicativity]\label{prop:Multiplicativity}
$\Phi$ preserves the multiplication, \ie $\Phi$ is a morphism of algebra objects in $\CHM_{\C}$.
\end{prop}
\begin{proof}
The cases of multiplying $\h(S)$ with itself or with a Lefschetz motive $\L_{x,\rho}$ are all obviously preserved by $\Phi$. We only need to show that for any $x\in S$ with non-trivial stabilizer $G_{x}$, the morphism $$\bigoplus_{\rho\in \Irr'(G_{x})}\L_{x, \rho}\to \bigoplus_{\substack{g\in G_{x}\\g\neq \id}}\L_{x,g}$$ given by the matrix with coefficient $\frac{1}{\sqrt{|G_{x}|}}\sqrt{\chi_{\rho_{0}}(g)-2}\cdot\chi_{\rho}(g)$ at place $(\rho, g)$ is multiplicative (note that the result of the multiplication could go outside of these direct sums to $\h(S)$). Since this is just a matrix computation, let us treat Lefschetz motives as 1-dimensional vector spaces (or equivalently, we are looking at the corresponding multiplicativity of the realization of $\Phi$ for Chow rings): let $E_{x,\rho}$ be the `generator' of $\L_{x,\rho}$ and $e_{x, g}$ be the `generator' of $\L_{x,g}$. Then the computations of the products in the previous two subsections say that:
\begin{eqnarray}
\label{eqn:E}
E_{x,\rho_{1}}\cdot E_{x,\rho_{2}}&=&\begin{cases}
-2x ~~~\text{  if  }  \rho_{1}=\rho_{2};\\
x ~~~~\text{  if  }  \rho_{1}, \rho_{2} \text{ are adjacent};\\
0 ~~~~\text{  if  }  \rho_{1},\rho_{2} \text{ are not adjacent};\\
\end{cases}\\
\label{eqn:e}
e_{x,g}\cdot e_{x,h}&=&\begin{cases}
x ~~~\text{  if  }  g=h^{-1};\\
0 ~~~~\text{  if  }  g\neq h^{-1};\\
\end{cases}
\end{eqnarray}

Therefore for any $\rho_{1}, \rho_{2}\in \Irr'(G_{x})$, we have 
\begin{eqnarray*}
&&\Phi(E_{x,\rho_{1}})\cdot \Phi(E_{x,\rho_{2}})\\
&=&\frac{1}{|G_{x}|}\sum_{g\in G_{x}}\sum_{h\in G_{x}}\sqrt{\chi_{\rho_{0}}(g)-2}\sqrt{\chi_{\rho_{0}}(h)-2}\,\chi_{\rho_{1}}(g)\chi_{\rho_{2}}(h)\, e_{x,g}\cdot e_{x,h}\\
&=&\frac{1}{|G_{x}|}\sum_{g\in G_{x}}\sqrt{\chi_{\rho_{0}}(g)-2}\sqrt{\chi_{\rho_{0}}(g^{-1})-2}\,\chi_{\rho_{1}}(g)\chi_{\rho_{2}}(g^{-1})\cdot x\\
&=&\frac{1}{|G_{x}|}\sum_{g\in G_{x}}(\chi_{\rho_{0}}(g)-2)\,\chi_{\rho_{1}}(g)\overline{\chi_{\rho_{2}}(g)}\cdot x\\
&=&\frac{1}{|G_{x}|}\left(\sum_{g\in G_{x}}\chi_{\rho_{0}\otimes \rho_{1}}(g)\overline{\chi_{\rho_{2}}(g)}-2\sum_{g\in G_{x}}\chi_{\rho_{1}}(g)\overline{\chi_{\rho_{2}}(g)}\right)\cdot x\\
&=& \left(\langle \rho_{0}\otimes\rho_{1}, \rho_{2}\rangle - 2\langle \rho_{1}, \rho_{2}\rangle \right)\cdot x\\
&=& \Phi\left(E_{x,\rho_{1}}\cdot E_{x,\rho_{2}}\right)
\end{eqnarray*}
where the first equality is the definition of $\Phi$ (and we add the non-existent $e_{x,1}$ with coefficient 0), the second equality uses (\ref{eqn:e}) the orthogonality among $e_{x,g}$'s (\ie $\L_{x,g}$'s), the third equality uses the fact that $\chi_{\rho_{0}}$ takes real value; the last equality uses all three cases of (\ref{eqn:E}).
\end{proof}

\begin{prop}[Additive isomorphism]\label{prop:AddIsom}
Taking $G$-invariants on both sides of (\ref{eqn:MKCorr}), $\Phi^{G}$ is an isomorphism of complex Chow motives between $\h(Y)$ and $\h_{orb}([S/G])$.
\end{prop}
\begin{proof}
We should prove the following morphism is an isomorphism:
\begin{equation*}
\Phi^{G}: \h(S)^{G}\oplus\left(\bigoplus_{x\in S}\bigoplus_{\rho\in \Irr'(G_{x})}\L_{x, \rho}\right)^{G} \to  \h(S)^{G}\oplus\left(\bigoplus_{x\in S}\bigoplus_{\substack{g\in G_{x}\\g\neq \id}}\L_{x,g}\right)^{G}.
\end{equation*}
Since $\Phi$ is given by `block diagonal matrix', it amounts to show that for each $x\in S$ (with $G_{x}$ non trivial),
the following is an isomorphism :
\begin{equation}\label{eqn:AddIsom}
\bigoplus_{\rho\in\Irr'(G_{x})}\L_{x,\rho}\to \left(\bigoplus_{\substack{g\in G_{x}\\g\neq \id}}\L_{x,g}\right)^{G_{x}}.
\end{equation}
which is equivalent to say that the following square matrix is non-degenerate\,: 
\begin{equation}\label{eqn:matrix}
\left(\sqrt{\chi_{\rho_{0}}(g)-2}\cdot\chi_{\rho}(g)\right)_{(\rho, [g])},
\end{equation}
where $\rho$ runs over the set $\Irr'(G_{x})$ of isomorphism classes of non-trivial irreducible representations and $[g]$ runs over the set of conjugacy classes of $G_{x}$ different from $\id$.

As this is about a matrix, it is enough to look at the realization of (\ref{eqn:AddIsom}):
\begin{equation*}
\bigoplus_{\rho\in\Irr'(G_{x})}E_{x,\rho}\to \left(\bigoplus_{\substack{g\in G_{x}\\g\neq \id}}e_{x,g}\right)^{G_{x}},
\end{equation*}
where both sides come equipped with non-degenerate quadratic forms given by intersection numbers and degrees of the orbifold product respectively. More precisely, by (\ref{eqn:E}) and (\ref{eqn:e}):
\begin{eqnarray*}
(E_{x,\rho_{1}}\cdot E_{x,\rho_{2}})&=&\begin{cases}
-2 ~~~\text{  if  }  \rho_{1}=\rho_{2};\\
1 ~~~~\text{  if  }  \rho_{1}, \rho_{2} \text{ are adjacent};\\
0 ~~~~\text{  if  }  \rho_{1},\rho_{2} \text{ are not adjacent};\\
\end{cases}\\
(e_{x,g}\cdot e_{x,h})&=&\begin{cases}
1 ~~~\text{  if  }  g=h^{-1};\\
0 ~~~~\text{  if  }  g\neq h^{-1};\\
\end{cases}
\end{eqnarray*}
which are both clearly non-degenerate. Now Proposition \ref{prop:Multiplicativity} shows that our matrix (\ref{eqn:matrix}) respects the non-degenerate quadratic forms on both sides, therefore it is non-degenerate.

Let us note here also an elementary proof which does not use the orbifold product. We first remark that for any $g\neq \id$, $\rho_{0}(g)\in \SL_{2}(\C)$ which is of finite order and different from the identity, hence its trace $\chi_{{0}}(g)\neq 2$.  Therefore the nondegeneracy of the matrix (\ref{eqn:matrix}) is equivalent to the nondegeneracy of the matrix $$\left(\chi_{\rho}(g)\right)_{(\rho, [g])},$$
which is obtained from the character table of the finite group $G_{x}$ by removing the first row (corresponding to the trivial representation) and the first column (corresponding to $\id\in G_{x}$). The nondegeneracy of this matrix is a completely general fact, which holds for all finite groups. We will give a proof in Lemma \ref{lemma:CharTable} at the end of this section.
\end{proof}

The combination of Lemma \ref{lemma:Equiv}, Proposition \ref{prop:Multiplicativity} and Proposition \ref{prop:AddIsom} proves the isomorphism of algebra objects (\ref{eqn:main}) in the main Theorem \ref{thm:main} in the global quotient case. For the isomorphisms for the Chow rings and cohomology rings, it is enough to apply realization functors. For the isomorphisms for the K-theory and topological K-theory, it suffices to invoke the construction of \emph{orbifold Chern characters} in \cite{MR2285746} which induce isomorphisms of algebras from (orbifold) K-theory to (orbifold) Chow ring as well as from (orbifold) topological K-theory to (orbifold) cohomology ring. The proof of Theorem \ref{thm:main} in the global quotient case is now complete. \hfill $\square$

The following lemma is used in the second proof of Proposition \ref{prop:AddIsom}. The elegant proof below is 
due to C\'edric Bonnaf\'e. We thank him for allowing us to use it. Recall that for a finite group $G$, its \emph{character table} is a square matrix whose rows are indexed by isomorphism classes of irreducible complex representations of $G$ and columns are indexed by conjugacy classes of $G$. 
\begin{lemma}\label{lemma:CharTable}
Let $G$ be any finite group. Then the matrix obtained from the character table by removing the first row corresponding to the trivial representation and the first column corresponding to the identity element, is non-degenerate.
\end{lemma}
\begin{proof}
Denote by $\1$ the trivial representation and by $\rho_{1}, \cdots, \rho_{n}$ the set of isomorphism classes of non-trivial representations of $G$. Suppose we have a linear combination $\sum_{i=1}^{n}c_{i}\chi_{\rho_{i}}$, with $c_{i}\in \C$, which vanishes for all non-identity conjugacy class, hence for all non-identity elements of $G$:
\begin{equation}\label{eqn:LinDep}
\sum_{i=1}^{n}c_{i}\chi_{\rho_{i}}(g)=0, ~~~~~~~\forall g\neq \id \in G.
\end{equation}
Set $$c_{0}\deff-\frac{1}{|G|}\sum_{i=1}^{n}c_{i}\dim(\rho_{i}),$$ and denote by $\chi_{reg}$ be the character of the regular representation, then (\ref{eqn:LinDep}) implies that the following linear combination vanishes for all $g\in G$: $$c_{0}\chi_{reg}+\sum_{i=1}^{n}c_{i}\chi_{\rho_{i}}=0.$$
If $c_{0}\neq 0$, it contradicts to the fact that the trivial representation should appear (with multiplicity 1) in the regular representation.\\
Hence we have $c_{0}=0$. Then by the linear independency among the characters of irreducible representations, we must have $c_{1}=\cdots=c_{n}=0$.
\end{proof}

\section{Proof of Theorem \ref{thm:main} : general orbifold case}\label{sect:ProofGeneral}
In this section, we give the proof of Theorem \ref{thm:main} in the full generality. As the proof goes essentially in the same way as the global quotient case in \S \ref{sect:ProofQuotient}, we will focus on the different aspects of the proof and refer to the arguments in \S \ref{sect:ProofQuotient} whenever possible.  

Recall the setting: $\rX$ is a two-dimensional Deligne--Mumford stack with only finitely many points with non-trivial stabilizers\,; $X$ is the underlying (projective) singular surface with only Du Val singularities and $Y\to X$ is the minimal resolution. For each $x\in X$, denote by $G_{x}$ its stabilizer, which is contained in  $\SL_{2}$.

Throughout this section, Chow groups of stacks are as in \cite{MR1005008} and Chow motives of stacks or singular $\Q$-varieties are as in \cite[\S 2]{MR1784411}.
\subsection{Resolution side}
Similar to (\ref{eqn:decompRes}), we have the following decomposition given by the classical McKay correspondence (see Introduction):
\begin{equation}
\h(Y)\isom \h(X)\oplus \bigoplus_{x\in X}\bigoplus_{\rho\in \Irr'(G_{x})}\L_{x, \rho}
\end{equation}
and the multiplication is the following:

\begin{itemize}
\item $\h(X)\otimes \h(X)\lra{\delta_{X}} \h(X)$ is the usual intersection product.
\item For any $\rho\in \Irr'(G_{x})$, $\h(X)\otimes \L_{x,\rho}\lra{i_{x}^{*}} \L_{x,\rho} $  is given by the class $x\in \CH^{2}(X)=\Hom(\h(X)\otimes \L, \L)$.
\item For any $\rho\in \Irr'(G_{x})$, $\L_{x,\rho}\otimes \L_{x,\rho}\lra{-2i_{x, *}} \h(X),$ is determined by $-2x\in\CH^{2}(X)$.
\item For any $\rho_{1}\neq \rho_{2}\in \Irr'(G_{x})$, 
\begin{itemize}
\item If they are \emph{adjacent}, that is, $\rho_{1}$ appears (with multiplicity 1) in the $G_{x}$-module $\rho_{2}\otimes \C^{2}$, where $\C^{2}$ is such that $\C^{2}/G_{x}$ is the singularity type of $x$, then $$\L_{x,\rho_{1}}\otimes \L_{x,\rho_{2}}\lra{i_{x, *}} \h(X),$$ is determined by $x\in\CH^{2}(X)$.
\item If they are not adjacent, then $\L_{x,\rho_{1}}\otimes \L_{x,\rho_{2}}\lra{0} \h(X)$ is the zero map.
\end{itemize}
\item The other multiplication maps are zero.
\end{itemize}

\subsection{Orbifold side}
Similar to (\ref{eqn:OrbMotDec}), we have 
\begin{equation}\label{eqn:DecompStack}
\h(\rX)=\h(X)\oplus\bigoplus_{x\in X}\left(\bigoplus_{\substack{g\in G_{x}\\g\neq \id}}\L_{x,g}\right)^{G_{x}},
\end{equation}
where the action of $G_{x}$ is by conjugacy.

Note that degree $0$ twisted stable maps with $3$ marked points to $\mathcal{X}$ are either untwisted stable maps to $\mathcal{X}$ or a twisted map to one of the stacky points of $\mathcal{X}$. 
In the latter case, the irreducible components of the moduli space around these twisted stable maps and the obstruction bundle are the same as those of the twisted stable maps to the orbifold $[\mathbb{C}^2/G]$. 
It is then clear that the orbifold product can be described as if $\mathcal{X}$ is a global quotient. 
Therefore the  orbifold product on $\h(\rX)$ is given by the following, via (\ref{eqn:DecompStack}):
\begin{itemize}
\item $\h(X)\otimes \h(X) \lra{\delta_{S}} \h(X)$ is the usual intersection product.
\item For all $g\in G_{x}$, $\h(X)\otimes \L_{x,g} \lra{i_{x}^{*}}  \L_{x,g}$ determined by the class of $x\in X$.
\item For all $g\in G_{x}$, $\L_{x,g}\otimes \L_{x,g^{-1}} \lra{i_{x, *}} \h(X)$ determined by the class of $x\in X$.
\item The other multiplication maps are zero.
\end{itemize}

\subsection{The multiplicative isomorphism}
Similar to (\ref{eqn:MKCorr}), we define

\begin{equation}
\phi: \h(X)\oplus\bigoplus_{x\in X}\bigoplus_{\rho\in \Irr'(G_{x})}\L_{x, \rho} \to  \h(X)\oplus\bigoplus_{x\in X}\bigoplus_{\substack{g\in G_{x}\\g\neq \id}}\L_{x,g},
\end{equation}
which is given by the following `block matrix':
\begin{itemize}
\item $\id: \h(X)\to \h(X)$;
\item For each $x\in X$ (with non-trivial stabilizer $G_{x}$), the morphism $$\bigoplus_{\rho\in \Irr'(G_{x})}\L_{x, \rho}\to \bigoplus_{\substack{g\in G_{x}\\g\neq \id}}\L_{x,g}$$ is the `matrix' with coefficient 
$\frac{1}{\sqrt{|G_{x}|}}\sqrt{\chi_{\rho_{0}}(g)-2}\cdot\chi_{\rho}(g)$ at place $(\rho, g)\in \Irr'(G_{x})\times (G_{x}\backslash \{\id\})$,  where $\chi$ denotes the character, $\rho_{0}$ is the natural 2-dimensional representation $\C^{2}$ of $G_{x}$ such that $\C^{2}/G_{x}$ is the singularity type of $x$. Note that $\rho_{0}(g)$ has determinant 1, hence its trace $\chi_{\rho_{0}}(g)$ is a real number.
\item The other morphisms are zero.
\end{itemize}
To conclude Theorem \ref{thm:main}, on the one hand, the same proof as in Proposition \ref{prop:Multiplicativity} shows that $\phi$ is multiplicative.
On the other hand, one sees immediately that $\phi$ factorizes through $$\h(X)\oplus\bigoplus_{x\in X}\left(\bigoplus_{\substack{g\in G_{x}\\g\neq \id}}\L_{x,g}\right)^{G_{x}}.$$ It is thus enough to show that the following induced map is an (additive) isomorphism:
\begin{equation*}
\psi: \h(X)\oplus\bigoplus_{x\in X}\bigoplus_{\rho\in \Irr'(G_{x})}\L_{x, \rho} \to  \h(X)\oplus\bigoplus_{x\in X}\left(\bigoplus_{\substack{g\in G_{x}\\g\neq \id}}\L_{x,g}\right)^{G_{x}},
\end{equation*}
However this follows from the proof of Proposition \ref{prop:AddIsom}, where one shows that (\ref{eqn:AddIsom}) is an isomorphism. The proof of Theorem \ref{thm:main} is complete. 
\qed

\bibliographystyle{amsplain}
\bibliography{biblio_fulie}

\end{document}